\documentclass[12pt]{amsart}
\usepackage{amssymb,latexsym, amsmath, amsxtra}
\usepackage{graphicx}
\newtheorem{conjecture}{Conjecture}
\newtheorem{theorem}{Theorem}
\newtheorem{lemma}{Lemma}

\begin{document}

\title{Moments of zeta and correlations of divisor-sums: V}
 
\author{Brian Conrey}
\address{American Institute of Mathematics, 360 Portage Ave, Palo Alto, CA 94306, USA and School of Mathematics, University of Bristol, Bristol BS8 1TW, UK}
\email{conrey@aimath.org}
\author{Jonathan P. Keating}
\address{School of Mathematics, University of Bristol, Bristol BS8 1TW, UK}
\email{j.p.keating@bristol.ac.uk}

\thanks{We gratefully acknowledge support under EPSRC Programme Grant EP/K034383/1
LMF: $L$-Functions and Modular Forms.  Research of the first author was also supported 
by the American Institute of Mathematics and by a grant from the National Science 
Foundation. JPK is grateful for the following additional support: a Royal Society Wolfson Research 
Merit Award, a Royal Society Leverhulme Senior Research Fellowship, a grant from the Air Force Office of Scientific Research, Air Force Material Command, 
USAF (number FA8655-10-1-3088), and ERC Advanced Grant 740900 (LogCorRM). He is also pleased to thank the American Institute of Mathematics for hospitality during two visits when work on this project was conducted.  Finally, we are grateful to the referees for their careful reading of the manuscript and for extremely helpful comments and suggestions.}

\date{\today}

\begin{abstract} In this series of papers 
we examine the calculation of the $2k$th  moment and shifted moments of the Riemann 
zeta-function on the critical line using long Dirichlet 
polynomials and divisor correlations.   The present paper completes the general study 
of what we call Type II sums which utilize a circle method framework 
and a convolution of shifted convolution sums
  to obtain all of the lower order terms in the asymptotic 
formula for the mean square along
 $[T,2T]$ of a Dirichlet polynomial of arbitrary length 
with divisor functions as coefficients.
\end{abstract}

\maketitle

\section{Introduction}
This paper is part V of a sequence devoted to understanding how to conjecture all of 
the integral moments of the Riemann zeta-function from a number theoretic perspective. The method is to approximate
$\zeta(s)^k$ by a long Dirichlet polynomial and then to compute the mean square of the Dirichlet polynomial (c.f.~[GG]). 
There are many off-diagonal terms and it is the care of these that is the concern of these papers. In
particular it is necessary to treat the off-diagonal terms by a method invented by Bogomolny and Keating [BK1, BK2]. 
Our perspective on
this method is that it is most properly viewed as a multi-dimensional Hardy-Littlewood circle method. 

In previous papers [CK1, CK3] we have developed a general method to calculate what were called type-I off-diagonal contributions in [BK1, BK2]; these are the off-diagonal terms usually considered in number-theoretic computations, e.g.~in [GG]. In parts II [CK2] and IV [CK4] we considered the simplest of the type-II off-diagonal terms (in the terminology of [BK1, BK2]).  These, somewhat unexpectedly, give a significant contribution in certain cases.  They have not previously been analysed systematically.  Our purpose here is to develop a general method for computing  all type-II off-diagonal terms.  
 
The  formula we obtain is in complete agreement with all of the main terms  
predicted by the recipe of [CFKRS] (and in particular, with the leading order term conjectured in [KS]). 
 
\section{Shifted moments}
 
Let $A $ and $B $ be two sets of cardinality $k$ containing ``shifts'' which may be thought of as  parameters of size $\ll \frac{1}{\log T}$
where $T$ is the basic large parameter near where we want to know the average of a product of $\zeta$-functions.  Let
$$D_A(s):=\prod_{\alpha\in A} \zeta(s+\alpha) =\sum_{n=1}^\infty \frac{\tau_A(n)}{n^s};$$
this implicitly defines the arithmetic functions $\tau_A(n)$.  
A basic question (the moment problem) is to evaluate 
\begin{eqnarray*}
\label{eqn:moment} 
I^\psi_{A,B}(T)=\int_0^\infty \psi\left(\frac t T \right) D_A(s)D_B(1-s) ~dt
\end{eqnarray*}
where $\psi$ is a smooth function with compact support, say $\psi \in C^\infty[1,2]$, and $s=1/2+it$.
The technique we are developing in this sequence of papers is to approach the  moment problem  through long Dirichlet polynomials. To this end we let
$$D_A(s;X) =\sum_{n\le X} \frac{\tau_A(n)}{n^s}$$
be the truncated Dirichlet series for $D_A(s)$. 
By Perron's formula for $\Re s > 1/4$ we have
\begin{eqnarray*}D_A(s;X)&=& \frac{1}{2\pi i}\int_{(1)} D_A(s+z) \frac{X^z}{z} ~dz\\
&=& D_A(s)+ R_A(s;X) +\frac{1}{2\pi i}\int_{(-\epsilon)} D_A(s+z) \frac{X^z}{z} ~dz
\end{eqnarray*}
where 
\begin{eqnarray*}
R_A(s;X)&=&\sum_{\alpha\in A} \operatornamewithlimits{Res}_{z=1-s-\alpha}
D_A(s+z) \frac{X^z}{z}  \\
&=& \sum_{\alpha\in A} D_{A\setminus\{\alpha\}}(1-\alpha) \frac{X^{1-s-\alpha}}{(1-s-\alpha)}.
\end{eqnarray*}
 In this sequence of papers  we consider  
\begin{eqnarray*}
I^\psi_{A,B}(T;X):=\int_0^\infty \psi\left(\frac t T\right) \big(D_{A}(s;X)-R_A(s;X)\big)\big(D_{B}(1-s;X) -R_B(1-s;X)\big)~dt \end{eqnarray*}
for various ranges of $X$. The expectation is that if $X>T^k$ then
this will be asymptotically  equal to $I_{A,B}^\psi(T)$.\begin{footnote} 
{In the earlier papers we did not include 
the terms $R_A$ and $R_B$ because their contributions were negligible since $X$ was not  large.  In general these terms do not create extra difficulties and so they will  be ignored here as well.}
\end{footnote}

We calculate this average in two different (conjectural) ways: the first is via the recipe and the second is via the delta method applied to the 
correlations of shifted divisor functions. We show that these two methods produce identical detailed main terms.

To use the recipe of [CFKRS] to conjecture a formula for $I_{A,B}^\psi(T;X)$, we start with  
$$ D_{A}(s;X)-R_A(s;X)=\frac{1}{2\pi i}
\int_{(\epsilon)} \frac{X^{w}}{w}D_{A_w}(s)~dw
$$
so that 
$$I^\psi_{A,B}(T;X)= \frac{1}{(2\pi i)^2}
\iint_{(\epsilon),(\epsilon)}\frac{X^{z+w}}{zw} I^\psi_{A_{w},B_{z}}(T)~dw~dz
$$
where 
$$A_w:=\{\alpha+w:\alpha\in A\},$$
i.e. the set $A$ but with all its elements shifted by $w$. 
From the recipe [CFKRS] we expect that
$$ I^\psi_{A,B}(T)=T\int_0^\infty \psi(t)\sum_{U\subset A, V\subset B\atop |U|=|V|}
 \left(\frac{tT}{2\pi}\right)^{-\sum_{\hat{\alpha}\in U\atop
\hat{\beta}\in V}(\hat{\alpha}+\hat{\beta}) }\mathcal B(A-U+V^-,B-V+U^-,1)~dt+o(T)
$$
where 
$$\mathcal B(A,B,s):=\sum_{n=1}^\infty \frac{\tau_A(n)\tau_B(n)}{n^s}.$$
 We have used an unconventional notation here; 
 by $A-U+V^-$ we mean the following: start with the set $A$ and remove the elements of $U$ and then include the negatives of 
the elements of $V$. We think of the process as ``swapping" equal numbers of elements between  $A$ and $B$; when elements are removed from $A$
and put into $B$ they first get multiplied by $-1$. We keep track of these swaps with our equal-sized subsets $U$ and $V$  of $A$ and $B$;
and when we refer to the ``number of swaps'' in a term we mean the cardinality $|U|$ of $U$ (or, since they are of equal size, of $V$). 
We insert this conjecture  and expect that
\begin{eqnarray} \label{eqn:basicII}
 I^\psi_{A,B}(T;X) &=&  T\int_0^\infty \psi(t)\frac{1}{(2\pi i)^2}
\iint_{(\epsilon),(\epsilon)}\frac{X^{z+w}}{zw}\sum_{U\subset A, V\subset B\atop |U|=|V|}
 \left(\frac{tT}{2\pi}\right)^{-\sum_{\hat{\alpha}\in U\atop
\hat{\beta}\in V}(\hat{\alpha}+w+\hat{\beta}+z) }\\
&&\qquad \qquad \times \mathcal B(A_w-U_w+V^-_z,B_z-V_z+U^-_w,1)~dw~dz~dt +o(T).
\nonumber
\end{eqnarray}
We have done a little simplification here: instead of writing $U\subset A_w$ we have written
$U\subset A$ and changed the exponent of $(tT/2\pi)$ accordingly.

Notice that there is a factor $(X/T^{|U|})^{w+z}$ in the previous equation. As mentioned above we refer to  $|U|$ as the number of ``swaps" in the recipe, and now we see
more clearly the role it plays; in the terms above for which  $X<T^{|U|}$   we move the path of 
integration in $w$ or $z$ to $+\infty$ so that the factor  $(X/T^{|U|})^{w+z}\to 0$ and the contribution of such a term is 0.
Thus, the size of $X$ determines how many ``swaps'' we must keep track of. 
To account for this we introduce a parameter $\ell$ defined by
$$T^\ell \le X  < T^{\ell+1}.$$
Then the above may be rewritten as 
\begin{eqnarray} \label{eqn:basicIIa}
 I^\psi_{A,B}(T;X) &=&  T\int_0^\infty \psi(t)\frac{1}{(2\pi i)^2}
\iint_{(\epsilon),(\epsilon)}\frac{X^{z+w}}{zw}\sum_{U\subset A, V\subset B\atop |U|=|V|\le \ell}
 \left(\frac{tT}{2\pi}\right)^{-\sum_{\hat{\alpha}\in U\atop
\hat{\beta}\in V}(\hat{\alpha}+w+\hat{\beta}+z) }\\
&&\qquad \qquad \times \mathcal B(A_w-U_w+V^-_z,B_z-V_z+U^-_w,1)~dw~dz~dt +o(T)
\nonumber
\end{eqnarray}
where we have restricted the sum to at most $\ell$ swaps.

Now we turn to the second approach via divisor correlations with the goal of obtaining this 
formula in a completely different way. 
In [CK1] and [CK3] we accomplished this in the situations where there were 0 or 1 swaps
(i.e. when $X<T^2$).   In [CK2] we considered two swaps but in a special case. 
In [CK4]  we looked at the general case of two swaps. In this paper, which is the final paper of the sequence,  we look at the general case with any number of swaps.   

As an extension  of the ideas in these papers, we have also begun to explore the analogous calculations for averages of ratios of the zeta function, specifically in the context of zero correlations [CK5, CK6].

The second method to obtain a conjecture for $I^\psi_{A,B}(T)$  will involve an intricate study of  convolutions of shifted divisor problems and  will occupy the rest of this paper. We begin that calculation by integrating  term-by-term to obtain
 \begin{eqnarray} \label{eqn:basicIII}   
\frac 1 T  \int_0^\infty \psi\left(\frac t T\right) D_{A}(s;X)D_{B}(1-s;X) ~dt
&=& \sum_{m,n\le X }\frac{\tau_A(m)\tau_B(n)\hat{\psi}\left(\frac{T}{2\pi}\log \frac m n \right)}{\sqrt{mn}} 
\\&=:& \mathcal O_{A;B}(T;X)   \nonumber \end{eqnarray}
 where $\hat{\psi}(x)=\int_{-\infty}^\infty \psi(t) e(-xt) ~dt ~ (=\int_{0}^\infty \psi(t) e(-xt) ~dt$
because of the support of $\psi$).
 Now  let us assume that $\ell\le k$ where $\ell $ is defined above.
We partition  $A$ and $B$ into $\ell$  non-empty sets $A=A_1\cup A_2\dots \cup A_\ell$ and 
$B=B_1\cup B_2\dots \cup B_\ell$. Then $\tau_A$ and $\tau_B$ are convolutions: 
$\tau_A=\tau_{A_1}*\tau_{A_2}*\dots *\tau_{A_\ell}$ and 
$\tau_B=\tau_{B_1}*\tau_{B_2}*\dots *\tau_{B_\ell}$.
For any such partition, the right hand side of (\ref{eqn:basicIII}) is equal to 
$$\mathcal O_{A_1,\dots,A_\ell;B_1,\dots,B_\ell}(T;X):=\sum_{m_1m_2\dots m_\ell \le X\atop
n_1n_2 \dots n_\ell\le X} \frac{\tau_{A_1}(m_1)\dots\tau_{A_\ell}(m_\ell)
 \tau_{B_1}(n_1)\dots \tau_{B_\ell}(n_\ell)}
 {\sqrt{m_1\dots m_\ell n_1\dots n_\ell}} \hat \psi\left(\frac{T}{2\pi}\log\frac{m_1\dots m_\ell}{n_1\dots n_\ell}\right).$$
 In other words $\mathcal O_{A;B}(T;X)=\mathcal O_{A_1,\dots,A_\ell;B_1,\dots,B_\ell}(T;X)$
 as long as $A$ and $B$ are  the disjoint unions of the $A_i$ and $B_j$.
Now we want to define a refinement of this sum. We impose a pairing  $A_j$ with $B_j$ 
and analyze this sum according to rational approximations to $m_j/n_j$.
In this way, the ordering of the sets $A_i, B_j$ now matters.  The eventual evaluation of $\mathcal O_{A,B}(T;X)$ will involve a sum of these pairings, which we describe in detail in the next section.

\section{Type II convolution sums}
There are various ways to decompose $A$ and $B$ and various ways to ``pair'' divisor functions $\tau_{A_i}$ and $\tau_{B_j}$ in preparation for the delta method.

More importantly, however, it turns out that there are various stratifications that also present themselves; basically one for each rational ``direction.''  If we ignore these 
then a simple application of the expected main terms from the  delta-method analysis will lead us to the wrong main terms.

At first sight it seems that when we do this we are counting the same terms repeatedly. However, we believe that our situation is an example of Manin's stratified subvarieties wherein counting solutions to  high dimensional diophantine equations often involves identifying a collection of subvarieties 
on each of which the solutions are counted separately (by the delta method for example). The point is that the main terms of the delta method do not always count all of the solutions.  This phenomenon was first identified in [FMT]; see, for example, [B] and [LT] for reviews of the subject.

Given $A_1, \dots ,A_\ell$ and $B_1,\dots,B_\ell$, the number of ways to pair each $A_n$ with a $B_m$ so that all are paired off is $\ell!$. 
Let us consider the pairing of $A_j$ with $B_j$. 
Now we  think of 
 $m_j/n_j$ as being approximated  by 
a rational number  $M_j/N_j$ with a small denominator  for each of $j=1,2,\dots ,\ell$ 
where $(M_j,N_j)=1$. In this way we get  subvarieties indexed by the   rational directions $M_j/N_j$ with 
$1\le j\le \ell$.   
We  will use all  directions $M_j/N_j$ subject to the natural conditions $(M_j,N_j)=1$ and 
\begin{eqnarray} \label{eqn:condition} M_1\dots M_\ell = N_1\dots N_\ell.
\end{eqnarray}
 We  sum over all of the terms with $m_j/n_j$ 
close to $M_j/N_j$.  We  introduce variables $h_j$
where, for a given $m_j, n_j,  M_j$ and $N_j$, we 
  define 
$$h_j := m_j N_j-n_j M_j.$$
 The rapid decay of $\hat{\psi}$ governs the 
ranges of all of the variables; see  below.

 We have
$$ n_1\dots n_\ell M_1\dots M_\ell=
 (m_1N_1-h_1)\dots  (m_\ell N_\ell-h_\ell) $$
so that for $M_1\dots M_\ell =N_1\dots N_\ell$ we have 
$$
\frac{ n_1\dots  n_\ell}{m_1\dots m_\ell}=
\big(1-\frac{h_1}{m_1N_1}\big)\dots\big(1 -\frac{h_\ell}{m_\ell N_\ell}\big)  
$$
and 
$$\log\frac{n_1\dots n_\ell}{m_1\dots m_\ell}=-\frac{h_1}{m_1N_1}-\dots -\frac{h_\ell}{m_\ell N_\ell}
 +O\big(\frac{h_ih_j}{m_1\dots m_\ell N_1 \dots N_\ell}\big).
$$
The error term is negligible so we 
can arrange the sum as 
\begin{eqnarray*} \label{eqn:start}
\sum_{  (M_j,N_j)=1\atop 
M_1\dots M_\ell =N_1\dots N_\ell}
\sum_{h_1,\dots h_\ell }
\sum_{{m_1\dots m_\ell\le X }\atop {(*_1), \dots (*_\ell) }}
\frac{\tau_{A_1}(m_1)\dots \tau_{A_\ell}(m_\ell) \tau_{B_1}(n_1)\dots \tau_{B_\ell}(n_\ell)}
{\sqrt{m_1\dots m_\ell n_1\dots n_\ell}}
\hat \psi \left(\frac{Th_1}{2\pi m_1N_1}+\dots +\frac{Th_\ell}{2\pi m_\ell N_\ell} \right)
\end{eqnarray*}
 where 
$$ (*_j): m_jN_j-n_jM_j=h_j. 
$$
We can replace $n_j$ in the denominator by $ \frac{m_jN_j}{M_j}$. Thus we are led to define
\begin{eqnarray}\nonumber \mathcal{O}_{A_1,\dots,A_\ell;B_1,\dots,B_\ell}^\dagger(T;X):=  && 
 \sum_{(M_j,N_j)=1\atop 
M_1\dots M_\ell =N_1\dots N_\ell}  
\sum_{h_1,\dots h_\ell }
\sum_{{m_1\dots m_\ell\le X }\atop {(*_1), \dots (*_\ell) }}
\frac{\tau_{A_1}(m_1)\dots \tau_{A_\ell}(m_\ell) \tau_{B_1}(n_1)\dots \tau_{B_\ell}(n_\ell)}
{m_1\dots m_\ell } \\&& \qquad  \qquad \times 
\hat \psi \left(\frac{Th_1}{2\pi m_1N_1}+\dots +\frac{Th_\ell}{2\pi m_\ell N_\ell} \right) .
\end{eqnarray}

Also, we define
\begin{eqnarray}  \mathcal{O}_{\ell}(T;X):= \label{eqn:start}&& 
\frac{1}{w_\ell}\sum_{A=A_1\cup\dots \cup A_\ell
\atop B=B_1\cup\dots \cup B_\ell}
 \mathcal{O}_{A_1,\dots,A_\ell;B_1,\dots,B_\ell}^\dagger(T;X),
 \end{eqnarray}
 where the 
 weight factor 
 $$w_\ell=\ell!^2 \ell^{2k-2\ell}$$
 in   $\mathcal O_\ell(T;X)$ will be explained in a later section.
Note that $\ell$ is defined in terms of  $T$ and $X$, so its inclusion in the notation is redundant. 
 Now we can state 
 \begin{conjecture} Suppose that $k\le \ell+1$ 
 and $T^\ell \le X < T^{\ell+1}$. Then for some $\delta>0$, 
 $$ I^\psi_{A,B}(T; X)= T\mathcal O_\ell(T;X) +O(T^{1-\delta}).$$
 \end{conjecture}
 
One way to view this paper is that it  gives evidence for this conjecture. In particular, in the next few sections we will  conjecturally understand 
$ \mathcal{O}_{A_1,\dots,A_\ell;B_1,\dots,B_\ell}^\dagger(T;X)$ 
by replacing the shifted divisor sums by what the delta-method leads us to expect for them. 
Then we evaluate the result and prove the rigorous theorem that our evaluation is precisely the quantity on the right-hand side of (\ref{eqn:basicII}).

\section{The case where $h_1\dots h_\ell\ne 0$}
Let us first look at the situation where none of the $h_j$ is 0.
The idea is to evaluate $ \mathcal{O}_{A_1,\dots,A_\ell;B_1,\dots,B_\ell}^\dagger(T;X)$  by replacing the $m_j$ by real variables $u_j$ while the $M_j, N_j$ and $h_j$ remain   integer-valued variables (and the $n_j$ are determined by the equations $*_j$). 

To do this 
we will replace the convolution sums by their averages, i.e.
\begin{eqnarray}\label{eqn:delta} 
\langle \tau_{A}(m)\tau_{B}(n)\rangle^{(*)}_{m= u}
\sim \frac{1}{M}\sum_{q=1}^\infty r_q(h)\langle \tau_A(m)e(mN/q)\rangle_{m= u}
\langle \tau_B(n)e(nM/q)\rangle_{n= \frac{uN}{M}}
\end{eqnarray}
where $r_q(h)$ is the Ramanujan sum (usually denoted $c_q(h)$) and 
$$\langle \tau_A(m)e(mN/q)\rangle_{m= u}
=\frac{1}{2\pi i}\int_{|w-1|=\epsilon}D_A(w,e\big(\frac{N}{q}\big))u^{w-1}~dw$$
 where 
 $$ D_A(w,e\big(\frac{N}{q}\big))= \sum_{n=1}^\infty\frac{\tau_A(n)e(nN/q)}{n^s}.$$
(In the above few lines we have  replaced  a sum 
$\sum_{n\le x} a_n f(n)$ by an integral $\int_1^x f(u) \langle a_n\rangle _{n=u} ~du $
where (in the handy physics notation) $\langle a_n\rangle_{n=u}$ denotes the average of $a_n$ when $n=u$ (the instantaneous rate of change of a good approximation to $\sum_{n\le u} a_n $ with respect to $u$). In our context
this may be expressed using $A(s)=\sum_{n=1}^\infty a_n n^{-s}$ and defining 
$ \langle a_n\rangle _{n=u}= \operatornamewithlimits{Res}_{|s-1|<\epsilon}u^{s-1}A(s)$
where we sum   the residues at all of the poles of $A(s)$ near $s=1$. )

Thus, we believe that 
\begin{eqnarray*}&&\sum_{M_j,N_j,(M_j,N_j)=1\atop N_1\dots N_\ell=M_1\dots M_\ell}\sum_{h_1,\dots, h_\ell\atop h_1 \dots h_\ell\ne 0 }
 \int_{u_1\dots u_\ell\le X  }\langle \tau_{A_1}(m_1)\tau_{B_1}(n_1)\rangle^{(*_1)}_{m_1= u_1}
\dots 
\langle \tau_{A_\ell}(m_\ell)\tau_{B_\ell}(n_\ell)\rangle^{(*_\ell)}_{m_\ell= u_\ell} \\&&
\qquad \qquad \qquad \qquad  \qquad\times 
\hat \psi \left(\frac{Th_1}{2\pi u_1N_1}+\dots\frac{Th_\ell}{2\pi u_\ell N_\ell} \right)~\frac{du_1}{u_1}\dots ~\frac{du_\ell}
{u_\ell}.
\end{eqnarray*}
  is, up to a power savings, equal to
\begin{eqnarray*}&&
\sum_{M_j,N_j,(M_j,N_j)=1\atop N_1\dots N_\ell=M_1\dots M_\ell}
\frac{1}{M_1\dots M_\ell} \sum_{q_j, h_j}r_{q_1}(h_1)\dots r_{q_\ell}(h_\ell)
  \int_{T^\ell\le u_1\dots u_\ell\le X }\\&&\qquad 
  \times \prod_{j=1}^\ell
 \frac{1}{2\pi i }\int_{|w_j-1|=\epsilon} 
 D_{A_j}(w_j,e\big(\frac{N_j}{q_j}\big))u_j^{w_j-1}dw_j\prod_{k=1}^\ell 
  \frac{1}{2\pi i }\int_{|z_k-1|=\epsilon} 
 D_{B_k}(z_k,e\big(\frac{M_k}{q_k}\big))\big(\frac {u_k N_k}{M_k}\big)^{z_k-1}dz_k
  \\&&\qquad \qquad \qquad \times
\hat{\psi}\left( T\bigg(\frac{h_1}{2\pi u_1N_1}+\dots +\frac{h_\ell }{2\pi u_\ell N_\ell}\bigg)\right) 
~\frac{du_1}{u_1}\dots \frac{du_\ell}{u_\ell}.
\end{eqnarray*}
 
 To further analyze this quantity, 
we make the changes of variable $v_j=\frac{T|h_j|}{2\pi u_jN_j}$  
  and bring the sums over the $h_j$ 
 to the inside; $u_1\dots u_\ell <X$ implies that 
$$ |h_1 \dots h_\ell| < \frac{(2\pi)^\ell Xv_1\dots v_\ell N_1\dots N_\ell }{T^\ell}.$$ 
We detect this condition using Perron's formula in  an integral over $s$. 
Then the above is 
\begin{eqnarray*}&& \int_{0<  v_1,\dots, v_\ell < \infty }
\sum_{M_j,N_j,(M_j,N_j)=1\atop N_1\dots N_\ell=M_1\dots M_\ell}
\frac{1}{M_1\dots M_\ell} \frac{1}{2\pi i}\int_{(2)}\sum_{q_j, h_j}\frac{r_{q_1}(h_1)\dots r_{q_\ell}(h_\ell)}
{|h_1\dots h_\ell|^s}
 \\&&\qquad 
  \times \prod_{j=1}^\ell
 \frac{1}{2\pi i }\int_{|w_j-1|=\epsilon} 
 D_{A_j}(w_j,e\big(\frac{N_j}{q_j}\big))\big(\frac{T|h_j|}{2\pi v_jN_j}\big)^{w_j-1}dw_j
 \\&&\qquad \qquad \prod_{k=1}^\ell 
  \frac{1}{2\pi i }\int_{|z_k-1|=\epsilon} 
 D_{B_k}(z_k,e\big(\frac{M_k}{q_k}\big))\big(\frac{T|h_k|}{2\pi v_kM_k}\big)^{z_k-1}dz_k
  \\&&\qquad \qquad \qquad \times \left(\frac{(2\pi)^\ell Xv_1\dots v_\ell N_1\dots N_\ell }{T^\ell}\right)^s
\hat{\psi}(\epsilon_1v_1+\dots +\epsilon_\ell v_\ell)
~\frac{dv_1}{v_1}\dots \frac{dv_\ell}{v_\ell}~\frac{ds}{s}
\end{eqnarray*}
where $\epsilon_j=\mbox{sgn}(h_j)$.
We simplify this a bit. We combine the middle two lines into a single product over $j$ and gather together all of the like variables (note that the sums over $h_j$ below are now restricted to the positive integers) :
\begin{eqnarray*}&&
\mbox{LHS}_\ell:= \frac{1}{2\pi i} \int_{(2)} X^s \big(\frac{T }{2\pi  }\big)^{-\ell s}
\sum_{{(M_1,N_1)=\dots =(M_\ell,N_\ell)=1\atop N_1\dots N_\ell=M_1\dots M_\ell}\atop \epsilon_j\in \{-1,+1\}} 
 \int_{0<  v_1,\dots, v_\ell < \infty }\hat{\psi}(\epsilon_1v_1+\dots +\epsilon_\ell v_\ell)  \\&&\qquad  \prod_{j=1}^\ell
\bigg[\frac{1}{(2\pi i)^2 }\iint_{ |w_j-1|=\epsilon\atop |z_j-1|=\epsilon} 
 M_j^{-z_j} N_j^{s+1-w_j} \sum_{h_j, q_j}\frac{r_{q_j}(h_j) }
{h_j ^{s+2-w_j-z_j}} v_j^{s+1-w_j-z_j}
 \\&&\qquad 
\qquad 
  D_{A_j}(w_j,e\big(\frac{N_j}{q_j}\big)) 
   D_{B_j}(z_j,e\big(\frac{M_j}{q_j}\big))\big(\frac{T }{2\pi  }\big)^{w_j+z_j-2}
 dw_j dz_j
~dv_j \bigg]~\frac{ds}{s}
\end{eqnarray*}

At this point we can rigorously identify $\mbox{LHS}_\ell$  with the terms on the right of  (\ref{eqn:basicII}),
through our key identity:
 \begin{theorem}
 \begin{eqnarray*}
 && \mbox{LHS}_\ell= \frac{1}{2\pi i} \int_{(2)} X^s  \int_0^\infty \psi(t)
 \sum_{U(\ell)\subset A\atop V(\ell) \subset B}
 \bigg(\frac {Tt}{2\pi}\bigg)^{-\sum_{\hat \alpha\in U(\ell)\atop \hat \beta \in V(\ell)}(\hat\alpha+\hat \beta+s)}
   \\&&\qquad \qquad 
 \times \mathcal B(A_s-U(\ell)_s+V(\ell)^-,B-V(\ell)+{U(\ell)_s}^-,1) ~dt 
 ~\frac{ds}{s}
\end{eqnarray*}
where $U(\ell)$ denotes a set of cardinality $\ell$ with precisely one element from each of $A_1,\dots ,A_\ell$ and similarly 
$V(\ell)$ denotes a set of cardinality $\ell$ with precisely one element from each of $B_1,\dots ,B_\ell$ .
\end{theorem}

\section{Preliminary reductions}
\begin{lemma}
\begin{eqnarray*}&&
\sum_{\epsilon_j \in \{-1,+1\}}\int_0^\infty \dots \int_0^\infty \hat{\psi}(\epsilon_1v_1+\dots\epsilon_\ell
v_\ell)  \prod_{j=1}^\ell v_j ^{s+1-w_j-z_j} dv_\ell\dots dv_1\\
&&\qquad =\int_0^\infty \psi(t) \prod_{j=1}^\ell \chi(w_j+z_j-s-1) t^{w_j+z_j-s-2} ~dt.
   \end{eqnarray*}
   \end{lemma}
   \begin{proof} The case $\ell=1$ of this identity may be found in [CK1]. 
   We may prove the general case by working our way from the inside out and using the technique of that proof.  For example, with fixed $v_1,\dots ,v_{\ell-1}, \epsilon_1, \dots ,\epsilon_{\ell-1}$  we have
  that the integral over $v_\ell$ is 
  \begin{eqnarray*}
  \int_{v_\ell=0}^\infty \int_{t=0}^\infty \psi(t) e(t\epsilon_1 v_1+\dots t\epsilon_{\ell-1}v_{\ell-1})
  (e(tv_\ell)+e(-tv_\ell)) v_\ell^{s+1-w_\ell-z_\ell} ~dt ~dv_\ell 
   \end{eqnarray*}
   We split this into two double integrals, one with $e(tv_\ell)$ and the other with $e(-tv_\ell)$.
   The first we rotate the $v_\ell$-path onto the positive imaginary axis, and the second we rotate the $v_\ell$ path onto the negative imaginary axis. By absolute convergence, we may now interchange the order of integration to arrive at a sum of two $v_\ell$-integrals inside a $t$-integral. We evaluate the $v_\ell$ integrals using the definition of the gamma-function. Then we repeat the process to evaluate the sum over $\epsilon_{\ell-1}$ of the integral over $v_{\ell-1}$ for a fixed 
   $v_1,\dots ,v_{\ell-2}, \epsilon_1, \dots ,\epsilon_{\ell-2}$. And so on. 
   \end{proof}

 \section{Poles}
 We have 
 \begin{eqnarray*}
 \langle \tau_A(m) e(m/q)\rangle_{m= u} \sim
\frac1 q \frac{1}{2\pi i}\int_{|w-1|=\frac 18}
 \prod_{\alpha\in A}\zeta(w+\alpha) G_A(w,q) \big(\frac u q\big)^{w-1} ~dw,
 \end{eqnarray*}
 where $G$ is a multiplicative function for which
$$G_A(1- \alpha,p^r)=\prod_{\hat \alpha \in A'}\left(1-\frac{1}{p^{1+\hat \alpha -\alpha}}\right)\sum_{j=0}^\infty \frac{\tau_{A'}(p^{j+r})}{p^{j(1- \alpha)}}$$
with $A'=A-\{ \alpha\}$ .
 With $*:mN-nM=h $, this leads to 
 \begin{eqnarray*}  &&
\langle \tau_{A}(m)\tau_{B}(n)\rangle^{(*)}_{m= u}
\sim \frac{1}{M}\sum_{q=1}^\infty r_q(h)\langle \tau_A(m)e(mN/q)\rangle_{m= u}
\langle \tau_B(n)e(nM/q)\rangle_{n= \frac{uN}{M}}\\&& \qquad \sim
\sum_{\alpha\in A\atop  \beta\in B}u^{-\alpha- \beta}M^{-1+\beta}N^{-\beta}
Z(A'_{-\alpha})Z(B'_{-\beta})  \sum_{d\mid h} \frac 1 {d^{1-\alpha- \beta}} 
\\&& \qquad \qquad \times
\sum_{q} \frac{\mu(q)(qd,M)^{1-\beta}
(qd,N)^{1- \alpha}}{q^{2-\alpha-\beta}} 
  G_{A}\big(1-\alpha,\frac{qd}{(qd,N)}\big)  
  G_{B}\big(1-\beta,\frac{qd}{(qd,M)}) ,\nonumber
\end{eqnarray*}
where  
$$Z(A)=\prod_{a\in A}\zeta(1+a).$$

 Inserting this into $\mbox{LHS}_\ell$ we have
 \begin{eqnarray*}&&
\mbox{LHS}_\ell= \frac{1}{2\pi i} \int_{(2)} X^s \big(\frac{T }{2\pi  }\big)^{-\ell s}
\sum_{{(M_1,N_1)=\dots =(M_\ell,N_\ell)=1\atop N_1\dots N_\ell=M_1\dots M_\ell} \atop \epsilon_j\in \{-1,+1\}}
 \int_{0<  v_1,\dots, v_\ell < \infty }\hat{\psi}(\epsilon_1 v_1+\dots +\epsilon_\ell v_\ell)  \\&& 
\sum_{U(\ell)\subset A\atop
V(\ell)\subset B} \prod_{j=1}^\ell
\bigg[  Z((A_j)'_{-\alpha_j})Z((B_j)'_{-\beta_j})
 M_j^{\beta_j-1} N_j^{s+\alpha_j} \sum_{h_j, q_j,d_j}\frac{1}
{h_j ^{s+\alpha_j+\beta_j}d_j^{1+s}}  \frac{\mu(q_j)(q_jd_j,M_j)^{1-\beta_j}
(q_jd_j,N_j)^{1- \alpha_j}}{q_j^{2-\alpha_j-\beta_j}} 
 \\&&\quad 
  v_j^{s-1+\alpha_j+\beta_j}
  G_{A_j}\big(1-\alpha_j,\frac{q_jd_j}{(q_jd_j,N_j)}\big)  
  G_{B_j}\big(1-\beta_j,\frac{q_jd_j}{(q_jd_j,M_j)})  
\big(\frac{T }{2\pi  }\big)^{-\alpha_j-\beta_j}
~dv_j \bigg]~\frac{ds}{s}
\end{eqnarray*}
where $U(\ell)=\{\alpha_1,\dots,\alpha_\ell\}$ with $\alpha_j\in A_j$ and $V(\ell)=\{\beta_1,\dots , \beta_\ell\}$ with $\beta_j\in B_j$.
Now we sum over the $h_j$ to get factors $\zeta(s+\alpha_j+\beta_j)$. Thus,
 \begin{eqnarray*}&&
\mbox{LHS}_\ell= \frac{1}{2\pi i} \int_{(2)} X^s \big(\frac{T }{2\pi  }\big)^{-\ell s}
\sum_{{(M_1,N_1)=\dots =(M_\ell,N_\ell)=1\atop N_1\dots N_\ell=M_1\dots M_\ell} \atop \epsilon_j\in \{-1,+1\}}
 \int_{0<  v_1,\dots, v_\ell < \infty }\hat{\psi}(\epsilon_1 v_1+\dots +\epsilon_\ell v_\ell)  
\\&&   \sum_{U(\ell)\subset A\atop
V(\ell)\subset B} \prod_{j=1}^\ell  \bigg[
  Z((A_j)'_{-\alpha_j})Z((B_j)'_{-\beta_j})
 M_j^{\beta_j-1} N_j^{s+\alpha_j} \zeta(s+\alpha_j+\beta_j)\\&& \qquad
 \sum_{ q_j,d_j}\frac{1}
{ d_j^{1+s}}  \frac{\mu(q_j)(q_jd_j,M_j)^{1-\beta_j}
(q_jd_j,N_j)^{1- \alpha_j}}{q_j^{2-\alpha_j-\beta_j}} 
 \\&&\quad 
  v_j^{s-1+\alpha_j+\beta_j}
  G_{A_j}\big(1-\alpha_j,\frac{q_jd_j}{(q_jd_j,N_j)}\big)  
  G_{B_j}\big(1-\beta_j,\frac{q_jd_j}{(q_jd_j,M_j)})  
\big(\frac{T }{2\pi  }\big)^{-\alpha_j-\beta_j}
~dv_j \bigg]~\frac{ds}{s}
\end{eqnarray*}
If we move the path of integration in $s$ to the line with $\Re s=\epsilon$, then we cross the poles of the $\zeta(s+\alpha_j+\beta_j)$ at $s=1-\alpha_j-\beta_j$. These contribute an amount that cancels the contribution of the $R_{A;B}(T;X)$.

 Next, we apply the lemma of Section 5  to evaluate the integral over the $v_j$ and obtain a factor of $ \chi(1-\alpha_j-\beta_j-s)$. 
 Then using  the functional equation for $\zeta$ we have  $\zeta(1-s-\alpha_j-\beta_j)$.
 Thus, the $s$-integrand without the $\frac{X^s}{s}$ in $\mbox{LHS}$ becomes
  \begin{eqnarray*}&&
   \int_0^\infty \psi(t)  \big(\frac{T }{2\pi  }\big)^{-\ell s}
\sum_{{(M_1,N_1)=\dots =(M_\ell,N_\ell)=1\atop N_1\dots N_\ell=M_1\dots M_\ell}  }
\sum_{U(\ell)\subset A\atop
V(\ell)\subset B} \prod_{j=1}^\ell
\zeta(1-s-\alpha_j-\beta_j) t^{-s-\alpha_j-\beta_j}
  \\&&\qquad \bigg[Z((A_j)'_{-\alpha_j})Z((B_j)'_{-\beta_j})
 M_j^{\beta_j-1} N_j^{s+\alpha_j} \sum_{q_j,d_j}\frac{1}
{d_j ^{1+s }}  \frac{\mu(q_j)(q_jd_j,M_j)^{1-\beta_j}
(q_jd_j,N_j)^{1- \alpha_j}}{q_j^{2-\alpha_j-\beta_j}} 
 \\&&\quad 
  G_{A_j}\big(1-\alpha_j,\frac{q_jd_j}{(q_jd_j,N_j)}\big)  
  G_{B_j}\big(1-\beta_j,\frac{q_jd_j}{(q_jd_j,M_j)})  
\big(\frac{T }{2\pi  }\big)^{-\alpha_j-\beta_j}
  \bigg] ~dt;
\end{eqnarray*}
our goal is to prove that this is equal to
 \begin{eqnarray*}
 &&   \int_0^\infty \psi(t)
 \sum_{U(\ell)\subset A\atop V(\ell) \subset B}
 \bigg(\frac {Tt}{2\pi}\bigg)^{-\sum_{\hat \alpha\in U(\ell)\atop \hat \beta \in V(\ell)}(\hat\alpha+\hat \beta+s)}
   \\&&\qquad \qquad 
 \times \mathcal B(A_s-U(\ell)_s+V(\ell)^-,B-V(\ell)+{U(\ell)_s}^-,1) ~dt. 
\end{eqnarray*}
 This further reduces to proving  for each $U(\ell), V(\ell)$ that
 \begin{eqnarray*}&&
 \sum_{{(M_1,N_1)=\dots =(M_\ell,N_\ell)=1\atop N_1\dots N_\ell=M_1\dots M_\ell}  }
 \prod_{j=1}^\ell
\zeta(1-s-\alpha_j-\beta_j)  
  \\&&\qquad \bigg[  Z((A_j)'_{-\alpha_j})Z((B_j)'_{-\beta_j})
 M_j^{\beta_j-1} N_j^{s+\alpha_j} \sum_{ q_j,d_j}\frac{1}
{d_j ^{1+s}}  \frac{\mu(q_j)(q_jd_j,M_j)^{1-\beta_j}
(q_jd_j,N_j)^{1- \alpha_j}}{q_j^{2-\alpha_j-\beta_j}} 
 \\&&\qquad \qquad 
  G_{A_j}\big(1-\alpha_j,\frac{q_jd_j}{(q_jd_j,N_j)}\big)  
  G_{B_j}\big(1-\beta_j,\frac{q_jd_j}{(q_jd_j,M_j)})  
  \bigg]\\&&=
 \mathcal B(A_s-U(\ell)_s+V(\ell)^-,B-V(\ell)+{U(\ell)_s}^-,1).
 \end{eqnarray*}

\section{Local considerations}

 We shall find it convenient to state our main theorem as an identity  
 of the Euler factor at a prime $p$. We begin by introducing 
   a set-theoretic  notation.
First of all, since $p$ is fixed for this discussion  we will often suppress it. In fact we write $X$ for $1/p$
and mostly consider power series in $X$. We take the unusual step of suppressing not only the prime $p$ 
but the divisor function and so we write $A(n)$ in place of $\tau_A(p^n)$. Also, for a set $A$ we let 
$$A_\alpha=\{a+\alpha:a\in A\}.$$
A further piece of notation: $A^+=A\cup\{0\}$. We have two important identities. The first is 
$$A^+(d)=A(d)+A^+(d-1).$$
This is a special case of
$$ (A\cup\{-\alpha\})(d-1) = X^{\alpha} \bigg( (A\cup\{-\alpha\})(d) - A(d)\bigg).$$
The other identity is 
$$\sum_{r=0}^R A(r+M)=  A^+(R+M)- A^+(M-1)$$
which follows by repeated application of the first identity. 

For arbitrary sets  $A$,$B$  we let
$$\mathcal \mathcal C(A,B):=\sum_{M=0}^\infty A(M)B(M)X^M.$$

Also, we let 
$$Z(A)=\sum_{j=0}^\infty A(j) X^j=\prod_{a\in A}(1-X^{1+a})^{-1}.$$

We begin with sets $A_j,B_j$  and numbers $\alpha_j,\beta_j$  for $j=1,2,\dots, \ell$.
We consider
\begin{eqnarray*}&&
\mathcal Q:= \sum_{\min(M_j,N_j)=0\atop
\sum_{j=1}^\ell M_j =\sum_{j=1}^\ell N_j}  
\prod_{j=1}^\ell \Sigma_{A_j,B_j,\alpha_j,\beta_j}(M_j,N_j) X^{ M_j(1-\beta_j)-N_j\alpha_j}
\end{eqnarray*}
where 
\begin{eqnarray*}
 \Sigma_{A,B,\alpha,\beta}(M,N)&=&\sum_{d,j,k\atop q\le 1  }
  (-1)^q X^{d(\alpha+\beta)}{A_{-\alpha}}(j+q+d-\min(q+d,N))\\&&\qquad \times {B_{-\beta}}(k+q+d-\min(q+d,M))  
  X^{2q+d+j+k-\min(q+d,M)-\min(q+d,N)}
\end{eqnarray*} 
 
Our identity is 
\begin{theorem}
$$\mathcal Q=\prod_{j=1}^\ell (1-X^{1-\alpha_j-\beta_j}) \mathcal C(A_1\cup\dots 
\cup A_\ell \cup\{-\beta_1,\dots, -\beta_\ell\},
B_1\cup\dots \cup  B_\ell \cup \{-\alpha_1,\dots,-\alpha_\ell\}).$$
\end{theorem}
 
 By the results of the previous section, Theorem 1 follows from Theorem 2 with $(A_j\setminus\{\alpha_j\})_s$ in place of $A_j$, $B_j\setminus\{\beta_j\}$ in place of $B_j$, and $\alpha_j+s$
in place of $\alpha_j$.

 \subsection{Some lemmas}
 Because of the condition $\min(M_j,N_j)=0$ 
 we  consider $\Sigma_{A,B,\alpha,\beta}(M,0)$ and $\Sigma_{A,B,\alpha,\beta}(0,N)$.
We have
 \begin{lemma}
 \begin{eqnarray*}
\Sigma_{A,B,\alpha,\beta}(M,0)&=&X^{M\beta} \bigg(
     \sum_{K }(B\cup\{-\alpha\})(K) A (K+M )  
  X^{K }
-  \sum_{K }B (K) {A }(K+M )  
  X^{K }\\
&&\qquad 
+\sum_K  {B }(K)  
  (A\cup\{-\beta\})(K+M )X^K\bigg)
\end{eqnarray*}
and
\begin{eqnarray*}
\Sigma_{A,B,\alpha,\beta}(0,N)&=&X^{N\alpha}\bigg(
    \sum_{K }(A\cup\{-\beta\})(K) B (K+N )  
  X^{K}
-  \sum_{K }A (K)B(K+N )  
  X^{K }\\
&&\qquad +\sum_K A(K)  
  (B\cup\{-\alpha\})(K+N )X^K\bigg).
\end{eqnarray*}

\end{lemma}
We defer the proof to later. 

The result of the lemma leads us to consider 
$$\mathcal F(A_1,\dots,A_\ell;B_1,\dots, B_\ell)=
\sum_{
{\sum M_i=\sum N_i  \atop 
\min(M_i,N_i)=0}\atop K_i}
\prod_{j=1}^\ell 
A_j(K_j+M_j)B_j(K_j+N_j)X^{K_j+M_j}.$$
 We will prove
\begin{lemma} 
We have
$$\mathcal F(A_1,\dots,A_\ell;B_1,\dots, B_\ell)=
\mathcal C(A_1\cup\dots \cup A_\ell, B_1\cup\dots \cup B_\ell ).
$$
\end{lemma}

The right-hand side of Theorem 2  may be expanded.  This leads to
\begin{lemma} For  $J\subset \{1,\dots,\ell\}$ let
$$-\beta_J=\{-\beta_j:j\in J\}\qquad -\alpha_J=\{-\alpha_j:j\in J\}.$$
We have 
\begin{eqnarray*}&&
\prod_{j=1}^\ell (1-X^{1-\alpha_j-\beta_j}) \mathcal C(A_1\cup\dots 
\cup A_\ell \cup\{-\beta_1,\dots, -\beta_\ell\},
B_1\cup\dots \cup  B_\ell \cup \{-\alpha_1,\dots,-\alpha_\ell\})\\
&& \qquad = 
(-1)^\ell\sum_{J_1, J_2\subset \{1,\dots,\ell\}
\atop J_1\cap J_2 =\emptyset} (-1)^{|J_1|+|J_2|} \mathcal C (A\cup -\beta_{J_1},B\cup -\alpha_{J_2}) 
\end{eqnarray*}
where $A=A_1\cup\dots\cup A_\ell$ and $B=B_1\cup\dots \cup B_\ell$.
\end{lemma}

The combination of these three lemmas easily leads to a proof of Theorem 2. 

\subsection{Proof of Theorem 2}
\begin{proof}
By Lemma 1 the left side of the identity in Theorem 2 may be written as 
\begin{eqnarray*}&&
  \sum_{\min(M_j,N_j)=0\atop
\sum_{j=1}^\ell M_j =\sum_{j=1}^\ell N_j}  
\prod_{j=1}^\ell \Sigma_{A_j,B_j,\alpha_j,\beta_j}(M_j,N_j) X^{ M_j(1-\beta_j)-N_j\alpha_j}\\\qquad&&
= \sum_{\min(M_j,N_j)=0\atop
\sum_{j=1}^\ell M_j =\sum_{j=1}^\ell N_j} X^{M_1+\dots M_\ell} \prod_{j=1}^\ell \bigg(
    \sum_{K }(A_j\cup\{-\beta_j\})(K+M_j) B_j (K+N_j )  
  X^{K}\\&&\qquad 
-  \sum_{K }A_j (K+M_j)B_j(K+N_j )  
  X^{K }\\
&&\qquad \qquad +\sum_K A_j(K+M_j)  
  (B_j\cup\{-\alpha_j\})(K+N_j )X^K\bigg).
\end{eqnarray*}
By Lemma 2 this is 
\begin{eqnarray*}
=(-1)^\ell\sum_{J_1, J_2\subset \{1,\dots,\ell\}
\atop J_1\cap J_2 =\emptyset} (-1)^{|J_1|+|J_2|} \mathcal C (A\cup -\beta_{J_1},B\cup -\alpha_{J_2}) 
\end{eqnarray*}
and by Lemma 3 this is 
\begin{eqnarray*}
=\prod_{j=1}^\ell (1-X^{1-\alpha_j-\beta_j}) \mathcal C(A_1\cup\dots 
\cup A_\ell \cup\{-\beta_1,\dots, -\beta_\ell\},
B_1\cup\dots \cup  B_\ell \cup \{-\alpha_1,\dots,-\alpha_\ell\})
\end{eqnarray*}
which is the right side of the identity in Theorem 2.

\end{proof}

\subsection{Proof of first lemma}
\begin{proof}

 Expanding the $q$-sum, we have
\begin{eqnarray*}&&
 =\sum_{d,j,k  }
    X^{d(\alpha+\beta)}{A_{-\alpha}}(j+d ) {B_{-\beta}}(k+d-\min(d,M))  
  X^{d+j+k-\min(d,M) }\\ && \qquad 
-
\sum_{d,j,k  }
    X^{d(\alpha+\beta)}{A_{-\alpha}}(j+1+d ) {B_{-\beta}}(k+1+d-\min(1+d,M))  
  X^{2+d+j+k-\min(1+d,M) }.
\end{eqnarray*} 
We split this into the terms with $d< M$ and those with $d\ge M$. We have
\begin{eqnarray*}
\Sigma^-(M,0)
&=&\sum_{j,k\atop 
d< M  }
    X^{d(\alpha+\beta)}{A_{-\alpha}}(j+d ) {B_{-\beta}}(k )  
  X^{j+k }\\ && \qquad 
-
\sum_{j,k\atop d < M  }
    X^{d(\alpha+\beta)}{A_{-\alpha}}(j+1+d ) {B_{-\beta}}(k )  
  X^{1+j+k  }
\\&=& Z(B_{-\beta})\bigg(
\sum_{j\atop 
d < M  }
    X^{d(\alpha+\beta)}{A_{-\alpha}}(j+d )  
  X^{j } 
-
\sum_{j\atop d < M  }
    X^{d(\alpha+\beta)}{A_{-\alpha}}(j+1+d )  
  X^{1+j }\bigg).
\end{eqnarray*}
The sum over $j$ telescopes so that this is 
\begin{eqnarray*}
\Sigma^-(M,0)
&=& Z(B_{-\beta})
\sum_{ 
d < M  }
    X^{d(\alpha+\beta)}{A_{-\alpha}}(d ) \\
&=& Z(B_{-\beta})
\sum_{ 
d < M  }
     {A_{\beta}}(d ) =  Z(B_{-\beta})(A_{\beta})^+(M-1).
\end{eqnarray*}

Next we consider 
\begin{eqnarray*}
\Sigma^+(M)&=&
 \sum_{j,k\atop d \ge M  }
    X^{d(\alpha+\beta)}{A_{-\alpha}}(j+d ) {B_{-\beta}}(k+d-M)  
  X^{d+j+k-M }\\ && \qquad 
-
\sum_{j,k\atop d \ge M  }
    X^{d(\alpha+\beta)}{A_{-\alpha}}(j+1+d ) {B_{-\beta}}(k+1+d-M)  
  X^{2+d+j+k-M }.
\end{eqnarray*}
We replace $d$ by $d+M$ and have 
\begin{eqnarray*}
\Sigma^+(M)&=&
 \sum_{j,k,d  }
    X^{(d+M)(\alpha+\beta)}{A_{-\alpha}}(j+d+M ) {B_{-\beta}}(k+d)  
  X^{d+j+k }\\ && \qquad 
-
\sum_{j,k,d  }
    X^{(d+M)(\alpha+\beta)}{A_{-\alpha}}(j+1+d+M ) {B_{-\beta}}(k+1+d)  
  X^{2+d+j+k }.
\end{eqnarray*}
Now the sum over $j$ and $k$ telescopes and we have 
\begin{eqnarray*}
\Sigma^+(M)&=&
 \sum_{j,d  }
    X^{(d+M)(\alpha+\beta)}{A_{-\alpha}}(j+d+M ) {B_{-\beta}}(d)  
  X^{d+j }\\ && \qquad 
+ \sum_{k,d  }
    X^{(d+M)(\alpha+\beta)}{A_{-\alpha}}(d+M ) {B_{-\beta}}(k+d)  
  X^{d+k }\\ && \qquad 
- \sum_{d  }
    X^{(d+M)(\alpha+\beta)}{A_{-\alpha}}(d+M ) {B_{-\beta}}(d)  
  X^{d }.\\ && \qquad
\end{eqnarray*}
 
We recognize a convolution in the first term and rewrite this as
\begin{eqnarray*}
\Sigma^+(M)&=&
 \sum_{r }
    X^{M(\alpha+\beta)}{A_{-\alpha}}(r+M ) (B_{\alpha})^+(r)  
  X^{r }\\ && \qquad 
+ \sum_{k,d  }
    X^{(d+M)(\alpha+\beta)}{A_{-\alpha}}(d+M ) {B_{-\beta}}(k+d)  
  X^{d+k }\\ && \qquad 
- \sum_{d  }
    X^{(d+M)(\alpha+\beta)}{A_{-\alpha}}(d+M ) {B_{-\beta}}(d)  
  X^{d }.\\ && \qquad 
\end{eqnarray*}
 The middle term here may be written as   
\begin{eqnarray*}
  \sum_{K  }
     {B_{-\beta}}(K)  
  X^{K } \sum_{d\le K}A_{\beta}(d+M )&=& \sum_{K  }
     {B_{-\beta}}(K)  
  X^{K } \left((A_{\beta})^+(K+M )-(A_{\beta})^+(M-1 )\right)\\
&=& \sum_K  {B_{-\beta}}(K)  
  (A_{\beta})^+(K+M )X^K-Z(B_{-\beta})(A_{\beta})^+(M-1 ).
\end{eqnarray*}
The second term of this cancels with $\Sigma^-(M,0)$
and so we have 
\begin{eqnarray*}
\Sigma_{A,B,\alpha,\beta}(M,0)&=&
    X^{M(\alpha+\beta)}\sum_{K }(B_{\alpha})^+(K) {A_{-\alpha}}(K+M )  
  X^{K }
- X^{M(\alpha+\beta)}\sum_{K }B_{\alpha}(K) {A_{-\alpha}}(K+M )  
  X^{K }\\
&&\qquad 
+\sum_K  {B_{-\beta}}(K)  
  (A_{\beta})^+(K+M )X^K.
\end{eqnarray*}
This may be rewritten as 
\begin{eqnarray*}
\Sigma_{A,B,\alpha,\beta}(M,0)&=&X^{M\beta} \bigg(
     \sum_{K }(B\cup\{-\alpha\})(K) A (K+M )  
  X^{K }
-  \sum_{K }B (K) {A }(K+M )  
  X^{K }\\
&&\qquad 
+\sum_K  {B }(K)  
  (A\cup\{-\beta\})(K+M )X^K\bigg).
\end{eqnarray*}
\end{proof}

By symmetry
\begin{eqnarray*}
\Sigma_{A,B,\alpha,\beta}(0,N)&=&X^{N\alpha}\bigg(
    \sum_{K }(A\cup\{-\beta\})(K) B (K+N )  
  X^{K}
-  \sum_{K }A (K)B(K+N )  
  X^{K }\\
&&\qquad +\sum_K A(K)  
  (B\cup\{-\alpha\})(K+N )X^K\bigg).
\end{eqnarray*}

\subsection{Proof of second lemma}
\begin{proof}  We prove more generally that
\begin{eqnarray*}&& \sum_{
{M_1+\dots+M_\ell=N_1+\dots +N_\ell  \atop
\min(M_1,N_1)=\dots = \min(M_\ell,N_\ell)=0}\atop K_1,\dots K_\ell}
 A_1(K_1+M_1)B_1(K_1+N_1)\dots 
A_\ell(K_\ell+M_\ell)B_\ell(K_\ell+N_\ell)X^{K_1+\dots +K_\ell +M_1+\dots
+M_\ell } \\&&\qquad \qquad \qquad  =
\sum_{N=0}^\infty (A_1*\dots *A_\ell)(N)(B_1*\dots *B_\ell)(N)X^N
\end{eqnarray*}
where the $A_i$, and $B_i$ are any functions on the natural numbers (i.e.
sequences) and $*$ just means the usual Cauchy convolution one encounters when
multiplying power series together.
It suffices to prove  
\begin{eqnarray*}&& \sum_{
\min(M_1,N_1)=\dots = \min(M_\ell,N_\ell)=0\atop K_1,\dots K_\ell}
 A_1(K_1+M_1)B_1(K_1+N_1)\dots 
A_\ell(K_\ell+M_\ell)B_\ell(K_\ell+N_\ell)\\&& \qquad \qquad \times Y^{2(K_1+\dots +K_\ell) +M_1+\dots
+M_\ell +N_1+\dots +N_\ell} 
e(\theta(M_1+\dots +M_\ell-N_1-\dots N_\ell))
\\&&\qquad \qquad \qquad  =
\sum_{M,N=0}^\infty (A_1*\dots *A_\ell)(M)(B_1*\dots *B_\ell)(N)Y^{M+N}e(\theta(M-N))
\end{eqnarray*}
as then our desired result follows upon integrating $\theta$ from 0 to 1 upon taking $X=Y^2$.
But now the left hand side is a product
\begin{eqnarray*}\prod_{j=1}^\ell \sum_{K,M,N\atop \min(M,N)=0}A_j(M+K)B_j(N+K)Y^{2K+M+N}e(\theta((M+K)-(N+K)))
\end{eqnarray*}
and the right hand side is a product
\begin{eqnarray*}
\prod_{j=1}^\ell \sum_{R=0}^\infty  A_j(R) (Ye(\theta))^R \sum_{
S=0}^\infty  B_j(S) (Ye(-\theta))^S.
\end{eqnarray*}
Therefore, it suffices to prove that
\begin{eqnarray*}
 \sum_{K,M,N\atop \min(M,N)=0}A(M+K)B(N+K)Y^{K+M}Z^{K+N} 
  =\sum_{R=0}^\infty  A(R)Y^R \sum_{
S=0}^\infty  B(S) Z^S.
\end{eqnarray*}
To do this, we consider the right hand side and order the double sum according to the minimum, call it $K$, of $R$ and $S$.  The right hand side may be rewritten as
\begin{eqnarray*}\sum_{K=0}^\infty \sum_{R,S\atop
\min(R,S)=K}A(R)Y^R B(S) Z^S.
\end{eqnarray*}
Replacing $R$ by $M+K$ and $S$ by $N+K$, we see that this is
exactly the left hand side.
\end{proof}

\subsection{Proof of third lemma}
\begin{proof}
Recall that 
$$X^{-\beta} (A\cup \{-\beta\})(n)=(A\cup\{-\beta\})(n+1)-A(n+1).$$
Using this we see that
\begin{eqnarray*}&&
(1-X^{1-\alpha-\beta}) \sum_{m,n=0}^\infty (A\cup \{-\beta\})(m) X^{\frac m 2} (B\cup\{-\alpha\})(n) X^{\frac n 2}
\\&&\qquad =
\sum_{m,n=0}^\infty (A\cup \{-\beta\})(m) X^{\frac m 2} (B\cup\{-\alpha\})(n) X^{\frac n 2}\\
\qquad &&-\sum_{m,n=0}^\infty X^{\frac{(m+1)}{2}+\frac{(n+1)}2}\bigg((A\cup\{-\beta\})(m+1)-A(m+1)\bigg)\bigg((B\cup \{-\alpha\})(n+1)-B(n+1)\bigg).
\end{eqnarray*}
In the last  line we can replace $m+1$ and $n+1$ by $m$ and $n$ since 
$(A\cup\{-\beta\})(0)=A(0)=1$ and similarly for $B$.
Multiplying out the last line and combining it with the line above we have 
\begin{eqnarray*}&&
\sum_{m,n=0}^\infty (A\cup \{-\beta\})(m) X^{\frac m 2}B(n) X^{\frac n 2} 
+\sum_{m,n=0}^\infty A(m) X^{\frac m 2} (B\cup\{-\alpha\})(n) X^{\frac n 2}\\\qquad \qquad &&
-\sum_{m,n=0}^\infty A(m)B(n) X^{\frac {m+n} 2}.
\end{eqnarray*}
Now the idea is to apply this to each $A_j, B_j, \alpha_j, \beta_j$. 
We have 
\begin{eqnarray*}&&
 (1-X^{1-\alpha_1-\beta_1})\dots  (1-X^{1-\alpha_\ell-\beta_\ell})
 \mathcal C(A_1\cup\{-\beta_1\}\cup\dots \cup A_\ell\cup\{-\beta_\ell\},B_1\cup\{-\alpha_1\}\cup\dots \cup B_\ell\cup\{-\alpha_\ell\})\\&&\qquad=\int_0^1 
 \sum_{m_1,n_1}(1-X^{1-\alpha_1-\beta_1})
 (A_1\cup \{-\beta_1\})(m_1) (X^{\frac 12}e(\theta))^{m_1} (B_1\cup\{-\alpha_1\})(n_1) (X^{\frac 12}e(-\theta))^{n_1} \\&&\qquad  \dots 
  \sum_{m_\ell,n_\ell}(1-X^{1-\alpha_\ell-\beta_\ell})
 (A_\ell\cup \{-\beta_\ell\})(m_\ell) (X^{\frac 12}e(\theta))^{ m_\ell} (B_\ell\cup\{-\alpha_\ell\})(n_\ell) 
 (X^{\frac 12 }e(-\theta))^{ n_\ell } d\theta.
 \end{eqnarray*}
We end up with
\begin{eqnarray*}&&\int_0^1
\prod_{j=1}^\ell  \sum_{m_j,n_j=0}^\infty (X^{\frac 12}e(\theta))^{ m_j  } (X^{\frac 12} e(-\theta))^{ n_j  } \bigg(
 (A_j\cup \{-\beta_j\})(m_j)  B_j(n_j)  \\&&\qquad \qquad 
+ A_j(m_j)  (B_j\cup\{-\alpha_j\})(n_j) - 
  A_j(m_j)B_j(n_j) \bigg) d\theta, 
\end{eqnarray*}
which is equal to 
\begin{eqnarray*}(-1)^\ell
\sum_{J_1, J_2\subset \{1,\dots,\ell\}
\atop J_1\cap J_2 =\emptyset} (-1)^{|J_1|+|J_2|} \mathcal C (A\cup -\beta_{J_1},B\cup -\alpha_{J_2}). 
\end{eqnarray*}
\end{proof}

 \section{Terms with some $h_j=0$}

 Suppose that we are in the situation where
 $$h_1\dots  h_{\ell'}\ne 0  \qquad  \mbox{and} \qquad h_{\ell'+1}=\dots =h_\ell= 0.$$
 Then for each $j> \ell'$ we have 
 $$m_jN_j=n_jM_j.$$
 Since $(M_j,N_j)=1$ this implies that
 $$m_j=\kappa_j M_j \qquad \mbox{and} \qquad n_j=\kappa_j N_j$$
 for some $\kappa_j$.
Then, our sum is 
\begin{eqnarray*}  &&
\sum_{{M_j, N_j\le Q\atop (M_j,N_j)=1}\atop 
{M_1\dots M_\ell =N_1\dots N_\ell  \atop \kappa_{\ell'+1}
 M_{\ell'+1}
\dots \kappa_\ell M_\ell\le X}}
\frac{\tau_{A_{\ell'+1}}(\kappa_{\ell'+1}M_{\ell'+1})\tau_{B_{\ell'+1}}(\kappa_{\ell'+1}N_{\ell'+1})
\dots \tau_{A_\ell}(\kappa_{\ell}M_{\ell})\tau_{B_\ell}(\kappa_{\ell}N_{\ell})
}{\kappa_{\ell'+1}\dots \kappa_\ell M_{\ell'+1}\dots M_\ell}
\\&& 
\sum_{h_1\dots h_{\ell'}\ne 0 }
\sum_{{m_1\dots m_{\ell'}\le \frac{X}{\kappa_{\ell'+1}M_{\ell'+1}
\dots  \kappa_\ell  M_\ell}}\atop {(*_1), \dots (*_{\ell'}) }}
\frac{\tau_{A_1}(m_1)\dots \tau_{A_{\ell'}}(m_{\ell'}) \tau_{B_1}(n_1)\dots \tau_{B_{\ell'}}(n_{\ell'})}
{m_1\dots m_{\ell'}}
\hat \psi \left(\frac{Th_1}{2\pi m_1N_1}+\dots +\frac{Th_{\ell'}}{2\pi m_{\ell'} N_{\ell'}} \right)
\end{eqnarray*}
 where 
$$ (*_j): m_jN_j-n_jM_j=h_j. 
$$
 
Now, as before,  we replace the convolution sums (*) by their averages, i.e.
\begin{eqnarray*}
 \int_{u_1\dots u_{\ell'}\le X'  }\langle \tau_{A_1}(m_1)\tau_{B_1}(n_1)\rangle^{(*_1)}_{m_1\sim u_1}
\dots 
\langle \tau_{A_{\ell'}}(m_{\ell'})\tau_{B_{\ell'}}(n_{\ell'})\rangle^{(*_{\ell'})}_{m_{\ell'}\sim u_{\ell'}} 
\hat \psi \left(\frac{Th_1}{2\pi u_1N_1}+\dots\frac{Th_{\ell'}}{2\pi u_{\ell'} N_{\ell'}} \right)~\frac{du_1}{u_1}\dots ~\frac{du_{\ell'}}
{u_{\ell'}}
\end{eqnarray*}
where $X'$ is defined by
$$ X'=\frac{X}{\kappa_{\ell'+1}M_{\ell'+1}\dots \kappa_{\ell}M_{\ell}};$$
here $*:mN-nM=h $. We expect by the delta-method [DFI] that 
\begin{eqnarray*}
\langle \tau_{A}(m)\tau_{B}(n)\rangle^{(*)}_{m\sim u}
\sim \frac{1}{M}\sum_{q=1}^\infty r_q(h)\langle \tau_A(m)e(mN/q)\rangle_{m\sim u}
\langle \tau_B(n)e(nM/q)\rangle_{n\sim \frac{uN}{M}}
\end{eqnarray*}
with
$$\langle \tau_A(m)e(mN/q)\rangle_{m\sim u}
=\frac{1}{2\pi i}\int_{|w-1|=\epsilon}D_A(w,e\big(\frac{N}{q}\big))u^{w-1}~dw.$$
 
So, we are led to 
\begin{eqnarray*}&&
\sum_{{M_j, N_j\le Q\atop (M_j,N_j)=1}\atop 
{M_1\dots M_\ell =N_1\dots N_\ell  \atop \kappa_{\ell'+1}
 M_{\ell'+1}
\dots \kappa_\ell M_\ell\le X}}
\frac{\tau_{A_{\ell'+1}}(\kappa_{\ell'+1}M_{\ell'+1})\tau_{B_{\ell'+1}}(\kappa_{\ell'+1}N_{\ell'+1})
\dots \tau_{A_\ell}(\kappa_{\ell}M_{\ell})\tau_{B_\ell}(\kappa_{\ell}N_{\ell})
}{\kappa_{\ell'+1}\dots \kappa_\ell M_{1}\dots M_\ell}
\\&& 
 \sum_{q_j, h_j}r_{q_1}(h_1)\dots r_{q_{\ell'}}(h_{\ell'})
  \int_{T^{\ell'}\le u_1\dots u_{\ell'}\le X' } \prod_{j=1}^{\ell'}
 \frac{1}{2\pi i }\int_{|w_j-1|=\epsilon} 
 D_{A_j}(w_j,e\big(\frac{N_j}{q_j}\big))u_j^{w_j-1}dw_j
\\&&\qquad 
  \times \prod_{k=1}^{\ell'}   \frac{1}{2\pi i }\int_{|z_k-1|=\epsilon} 
 D_{B_k}(z_k,e\big(\frac{M_k}{q_k}\big))\big(\frac {u_k N_k}{M_k}\big)^{z_k-1}dz_k
  \\&&\qquad \qquad \qquad \times
\hat{\psi}\left( T\bigg(\frac{h_1}{2\pi u_1N_1}+\dots +\frac{h_{\ell'} }{2\pi u_{\ell'} N_{\ell'}}\bigg)\right) 
~\frac{du_1}{u_1}\dots \frac{du_{\ell'}}{u_{\ell'}}.
\end{eqnarray*}
 
We make the changes of variable $v_j=\frac{T|h_j|}{2\pi u_jN_j}$  for $1\le j\le \ell'$
  and bring the sums over the $h_j$ 
 to the inside; $u_1\dots u_{\ell'} <X'$ implies that 
$$ |h_1 \dots h_{\ell'}| < \frac{(2\pi)^{\ell'} X'v_1\dots v_{\ell'} N_1\dots N_{\ell'} }{T^{\ell'}}.$$ 
We detect this condition using Perron's formula in  an integral over $s$. 
Then the above is 
\begin{eqnarray*}&& \int_{0<  v_1,\dots, v_{\ell'} < \infty }
\sum_{{M_j, N_j\le Q\atop (M_j,N_j)=1}\atop 
{M_1\dots M_\ell =N_1\dots N_\ell  \atop \kappa_{\ell'+1}
 M_{\ell'+1}
\dots \kappa_\ell M_\ell\le X}}
\frac{\tau_{A_{\ell'+1}}(\kappa_{\ell'+1}M_{\ell'+1})\tau_{B_{\ell'+1}}(\kappa_{\ell'+1}N_{\ell'+1})
\dots \tau_{A_\ell}(\kappa_{\ell}M_{\ell})\tau_{B_\ell}(\kappa_{\ell}N_{\ell})
}{\kappa_{\ell'+1}\dots \kappa_\ell M_{1}\dots M_\ell} \\&&\qquad 
  \times\frac{1}{2\pi i}\int_{(2)}\sum_{q_j, h_j}\frac{r_{q_1}(h_1)\dots r_{q_{\ell'}}(h_{\ell'})}
{(h_1\dots h_{\ell'})^s}
 \prod_{j=1}^{\ell'}
 \frac{1}{2\pi i }\int_{|w_j-1|=\epsilon} 
 D_{A_j}(w_j,e\big(\frac{N_j}{q_j}\big))\big(\frac{Th_j}{2\pi v_jN_j}\big)^{w_j-1}dw_j
 \\&&\qquad \qquad \prod_{k=1}^{\ell'} 
  \frac{1}{2\pi i }\int_{|z_k-1|=\epsilon} 
 D_{B_k}(z_k,e\big(\frac{M_k}{q_k}\big))\big(\frac{Th_k}{2\pi v_kM_k}\big)^{z_k-1}dz_k
  \\&&\qquad \qquad \qquad \times \left(\frac{(2\pi)^{\ell'} X'v_1\dots v_{\ell'} N_1\dots N_{\ell'} }
  {T^{\ell'}}\right)^s \sum_{\epsilon_j\in \{-1,+1\}}
\hat{\psi}(\epsilon_1v_1+\dots +\epsilon_{\ell'} v_{\ell'})
~\frac{dv_1}{v_1}\dots \frac{dv_{\ell'}}{v_{\ell'}}~\frac{ds}{s}
\end{eqnarray*}
where the $\epsilon_j=\mbox{sgn}(h_j)$ arise from taking account of the signs of the $h_j$.
We now simplify this a bit. We combine the middle two lines into a single product over $j$ and we gather together all of the like variables:
\begin{eqnarray*}&&
\mbox{LHS}_{\ell'}:= \frac{1}{(2\pi i)^2} \iint_{(2)} X^{s+s'} \big(\frac{T }{2\pi  }\big)^{-{\ell'} s}
\sum_{{(M_1,N_1)=\dots =(M_\ell,N_\ell)=1\atop N_1\dots N_\ell=M_{1}\dots M_\ell}\atop \epsilon_j\in \{-1,+1\}} \sum_{\kappa_{\ell'+1},\dots, \kappa_\ell}  \prod_{i=\ell'+1}^\ell \frac {\tau_{A_i}(M_i\kappa_i)\tau_{B_i}(N_i \kappa_i)}
{(\kappa_i M_i)^{1+s'+s}}   \\&&\qquad  
 \int_{0<  v_1,\dots, v_{\ell'} < \infty }\hat{\psi}(\epsilon_1v_1+\dots +\epsilon_{\ell'} v_{\ell'}) \prod_{j=1}^{\ell'} 
\bigg[\frac{1}{(2\pi i)^2 }\iint_{ |w_j-1|=\epsilon\atop |z_j-1|=\epsilon} 
 M_j^{-z_j} N_j^{s+1-w_j} \sum_{h_j, q_j}\frac{r_{q_j}(h_j) }
{h_j ^{s+2-w_j-z_j}} v_j^{s+1-w_j-z_j}
 \\&&\qquad 
\qquad 
  D_{A_j}(w_j,e\big(\frac{N_j}{q_j}\big)) 
   D_{B_j}(z_j,e\big(\frac{M_j}{q_j}\big))\big(\frac{T }{2\pi  }\big)^{w_j+z_j-2}
 dw_j dz_j
~dv_j \bigg]~\frac{ds}{s}~\frac{ds'}{s'}
\end{eqnarray*}

 Now we have another  key identity:
 \begin{theorem}
 \begin{eqnarray*}
 && \operatornamewithlimits{Res}_{s'=0} \mbox{LHS}_{\ell'}= \frac{1}{2\pi i} \int_{(2)} X^s  \int_0^\infty \psi(t)
 \sum_{U(\ell')\subset A\atop V(\ell') \subset B}
 \bigg(\frac {Tt}{2\pi}\bigg)^{-\sum_{\hat \alpha\in U(\ell')\atop \hat \beta \in
  V(\ell')}(\hat\alpha+\hat \beta+s)}
   \\&&\qquad \qquad 
 \times \mathcal B(A_s-U(\ell')_s+V(\ell')^-,B-V(\ell')+{U(\ell')_s}^-,1) ~dt 
 ~\frac{ds}{s}
\end{eqnarray*}
where $U(\ell')$ denotes a set of cardinality $\ell'$ with precisely one element from each of $A_1,\dots ,A_{\ell'}$ and similarly 
$V(\ell')$ denotes a set of cardinality $\ell'$ with precisely one element from 
each of $B_1,\dots ,B_{\ell'}$ .
\end{theorem}

\section{Preliminary reductions, again}

The result of Section 5 implies that
\begin{eqnarray*}&&
\sum_{\epsilon_j \in \{-1,+1\}}\int_{v_1,\dots ,v_{\ell'}=0}^\infty\hat{\psi}(\epsilon_1v_1+\dots
\epsilon_{\ell'}
v_{\ell'})  \prod_{j=1}^{\ell'} v_j ^{s+1-w_j-z_j} dv_j\\
&&\qquad =\int_0^\infty \psi(t) \prod_{j=1}^{\ell'} \chi(w_j+z_j-s-1) t^{w_j+z_j-s-2} ~dt.
   \end{eqnarray*}

 \section{Poles, again}
As before we use 
 \begin{eqnarray*}  &&
\langle \tau_{A}(m)\tau_{B}(n)\rangle^{(*)}_{m= u}
\sim  
\sum_{\alpha\in A\atop  \beta\in B}u^{-\alpha- \beta}M^{-1+\beta}N^{-\beta}
Z(A'_{-\alpha})Z(B'_{-\beta})  \sum_{d\mid h} \frac 1 {d^{1-\alpha- \beta}} 
\\&& \qquad \qquad \times
\sum_{q} \frac{\mu(q)(qd,M)^{1-\beta}
(qd,N)^{1- \alpha}}{q^{2-\alpha-\beta}} 
  G_{A}\big(1-\alpha,\frac{qd}{(qd,N)}\big)  
  G_{B}\big(1-\beta,\frac{qd}{(qd,M)}) .\end{eqnarray*}
  
 Inserting this into $\mbox{LHS}_{\ell'}$ we have
 \begin{eqnarray*}&&
\mbox{LHS}_{\ell'}= 
\frac{1}{(2\pi i)^2} \iint_{(2)} X^{s+s'} \big(\frac{T }{2\pi  }\big)^{-{\ell'} s}
\sum_{{(M_1,N_1)=\dots =(M_\ell,N_\ell)=1\atop N_1\dots N_\ell=M_{1}\dots M_\ell}\atop \epsilon_j\in \{-1,+1\}} \sum_{\kappa_{\ell'+1},\dots ,\kappa_\ell}  \prod_{i=\ell'+1}^\ell \frac {\tau_{A_i}(M_i\kappa_i)\tau_{B_i}(N_i \kappa_i)}
{(\kappa_i M_i)^{1+s'+s}}   \\&&\qquad  
 \int_{0<  v_1,\dots, v_{\ell'} < \infty }\hat{\psi}(\epsilon_1v_1+\dots +\epsilon_{\ell'} v_{\ell'}) 
\prod_{j=1}^{\ell'}
\bigg[ \sum_{\alpha_j\in A_j\atop
\beta_j\in B_j}   v_j^{s-1+\alpha_j+\beta_j} 
 M_j^{\beta_j-1} N_j^{s+\alpha_j}\\&&\qquad \qquad \times  
 Z((A_j)'_{-\alpha_j})Z((B_j)'_{-\beta_j})
 \sum_{h_j, q_j,d_j}\frac{1}
{h_j ^{s+\alpha_j+\beta_j}d_j^{1+s}}  \frac{\mu(q_j)(q_jd_j,M_j)^{1-\beta_j}
(q_jd_j,N_j)^{1- \alpha_j}}{q_j^{2-\alpha_j-\beta_j}}
\\ &&\qquad
   G_{A_j}\big(1-\alpha_j,\frac{q_jd_j}{(q_jd_j,N_j)}\big)  
  G_{B_j}\big(1-\beta_j,\frac{q_jd_j}{(q_jd_j,M_j)})  
\big(\frac{T }{2\pi  }\big)^{-\alpha_j-\beta_j}
~dv_j \bigg]~\frac{ds}{s}~\frac{ds'}{s'}.
\end{eqnarray*}
Now we sum over the $h_j$ to get factors $\zeta(s+\alpha_j+\beta_j)$;
these pair up with the factors $\chi(w_j+z_j-s-1)$ which turned into $ \chi(1-\alpha_j-\beta_j-s)$
after collecting the residues $w_j=1-\alpha_j$ and $z_j=1-\beta_j$ that arose from the integral over  $v_j$. 
 Then using  the functional equation for $\zeta$ we have  $\zeta(1-s-\alpha_j-\beta_j)$.
 Thus, the $s$-integrand without the $\frac{X^s}{s}$ in $\mbox{LHS}_{\ell'}$ becomes
  \begin{eqnarray*}&&
 \int_0^\infty \psi(t) \frac{1}{2\pi i}\int_{(2)} X^{s'}\big(\frac{T }{2\pi  }\big)^{-\ell' s}
\sum_{{(M_1,N_1)=\dots =(M_\ell,N_\ell)=1\atop N_1\dots N_\ell=M_1\dots M_\ell}  }
\sum_{\kappa_{\ell'+1},\dots ,\kappa_\ell}  \prod_{i=\ell'+1}^\ell
 \frac {\tau_{A_i}(M_i\kappa_i)\tau_{B_i}(N_i \kappa_i)}
{(\kappa_i M_i)^{1+s'+s}} \\&&\qquad  \sum_{U(\ell')\subset A\atop 
V(\ell')\subset B} 
 \prod_{j=1}^{\ell'}  \bigg[Z((A_j)'_{-\alpha_j})Z((B_j)'_{-\beta_j})
\zeta(1-s-\alpha_j-\beta_j) t^{-s-\alpha_j-\beta_j}
  \\&&\qquad \qquad \times 
 M_j^{\beta_j-1} N_j^{s+\alpha_j} \sum_{q_j,d_j}\frac{1}
{d_j ^{1+s }}  \frac{\mu(q_j)(q_jd_j,M_j)^{1-\beta_j}
(q_jd_j,N_j)^{1- \alpha_j}}{q_j^{2-\alpha_j-\beta_j}} 
 \\&&\quad \qquad \qquad  \times
  G_{A_j}\big(1-\alpha_j,\frac{q_jd_j}{(q_jd_j,N_j)}\big)  
  G_{B_j}\big(1-\beta_j,\frac{q_jd_j}{(q_jd_j,M_j)})  
\big(\frac{T }{2\pi  }\big)^{-\alpha_j-\beta_j}
  \bigg] \frac{ds'}{s'} ~dt, 
\end{eqnarray*}
where $U(\ell')=\{\alpha_1,\dots,\alpha_{\ell'}\}$ and $V(\ell')=\{\beta_1,\dots,\beta_{\ell'}\}$. 
Our goal is to prove that the residue of this at $s'=0$  is equal to
 \begin{eqnarray*}
 &&   \int_0^\infty \psi(t)
 \sum_{U(\ell')\subset A\atop V(\ell') \subset B}
 \bigg(\frac {Tt}{2\pi}\bigg)^{-\sum_{\hat \alpha\in U(\ell')\atop \hat \beta \in V(\ell')}
 (\hat\alpha+\hat \beta+s)}
   \\&&\qquad \qquad 
 \times \mathcal B(A_s-U(\ell')_s+V(\ell')^-,B-V(\ell')+{U(\ell')_s}^-,1) ~dt. 
\end{eqnarray*}
 This further reduces to proving that
 \begin{eqnarray*}&& \operatornamewithlimits{Res}_{s'=0}  
 \sum_{{(M_1,N_1)=\dots =(M_\ell,N_\ell)=1\atop N_1\dots N_\ell=M_1\dots M_\ell}  }
 \sum_{\kappa_{\ell'+1},\dots ,\kappa_\ell}  \prod_{i=\ell'+1}^\ell \frac {\tau_{A_i}(M_i\kappa_i)\tau_{B_i}(N_i \kappa_i)}
{(\kappa_i M_i)^{1+s'+s}} \prod_{j=1}^{\ell'}
\zeta(1-s-\alpha_j-\beta_j)  
  \\&&\qquad \bigg[ 
 M_j^{\beta_j-1} N_j^{s+\alpha_j} \sum_{ q_j,d_j}\frac{1}
{d_j ^{1+s }}  \frac{\mu(q_j)(q_jd_j,M_j)^{1-\beta_j}
(q_jd_j,N_j)^{1- \alpha_j}}{q_j^{2-\alpha_j-\beta_j}} 
 \\&&\qquad \qquad Z((A_j)'_{-\alpha_j})Z((B_j)'_{-\beta_j})
  G_{A_j}\big(1-\alpha_j,\frac{q_jd_j}{(q_jd_j,N_j)}\big)  
  G_{B_j}\big(1-\beta_j,\frac{q_jd_j}{(q_jd_j,M_j)})  
  \bigg]
  \frac{X^{s'}}{s'}
  \\&&=
 \mathcal B(A_s-U(\ell')_s+V(\ell')^-,B-V(\ell')+{U(\ell')_s}^-,1). \end{eqnarray*}

\section{Local considerations, again}
Again  we convert the above to an identity about the Euler factor of each side at a prime $p$.

 With arbitrary sets $A_j,B_j$  and numbers $\alpha_j,\beta_j$  for $j=1,2,\dots, \ell'$,
we consider
\begin{eqnarray*}&&
\mathcal Q':= \sum_{{\kappa_{\ell'+1},\dots, \kappa_\ell\atop\min(M_j,N_j)=0}\atop
\sum_{j=1}^\ell M_j =\sum_{j=1}^\ell N_j} 
\prod_{i=\ell'+1}^{\ell}  A_i(M_i+\kappa_i) B_i(N_i+\kappa_i)X^{M_i+\kappa_i}
\\&&\qquad \qquad \times
\prod_{j=1}^{\ell'} \Sigma_{A_j,B_j,\alpha_j,\beta_j}(M_j,N_j) X^{ M_j(1-\beta_j)-N_j\alpha_j},
\end{eqnarray*}
where 
\begin{eqnarray*}
 \Sigma_{A,B,\alpha,\beta}(M,N)&=&\sum_{d,j,k\atop q\le 1  }
  (-1)^q X^{d(\alpha+\beta)}{A_{-\alpha}}(j+q+d-\min(q+d,N))\\&&\qquad \times {B_{-\beta}}(k+q+d-\min(q+d,M))  
  X^{2q+d+j+k-\min(q+d,M)-\min(q+d,N)},
\end{eqnarray*} 
 as before.
 
Our identity is 
\begin{theorem}
$$\mathcal Q'=\prod_{j=1}^{\ell'} (1-X^{1-\alpha_j-\beta_j}) \mathcal C(A_1\cup\dots 
\cup A_\ell \cup\{-\beta_1,\dots, -\beta_{\ell'}\},
B_1\cup\dots \cup  B_\ell \cup \{-\alpha_1,\dots,-\alpha_{\ell'}\}).$$
\end{theorem}
 By the results of the previous section, Theorem 3 follows from Theorem 4 with $(A_i)_s$ in place of $A_i$ for $i=\ell'+1,\dots,\ell$ and with 
 $(A_j\setminus\{\alpha_j\})_s$, $B_j\setminus \{\beta_j\}$ and $\alpha_j+s$ in place of $A_j$, $B_j$, and $\alpha_j$, respectively, for $j=1,\dots,\ell'$.

 \subsection{Recall lemmas}
 Our earlier  lemma implies that if  $\min(M,N)=0$ then
 \begin{eqnarray*}
 \Sigma_{A,B,\alpha,\beta}(M,N)&=&X^{M\beta+N\alpha}\bigg(
    \sum_{K }(A\cup\{-\beta\})(K+M) B (K+N )  
  X^{K} \\&& \qquad 
-  \sum_{K }A (K+M)B(K+N )  
  X^{K }\\
&&\qquad +\sum_K A(K+M)  
  (B\cup\{-\alpha\})(K+N )X^K\bigg).
\end{eqnarray*}
 So, we can replace the $\Sigma$s in the   formula for $\mathcal Q'$ by this expression.
 
 Thus, we have
 \begin{eqnarray*}&&
\mathcal Q'= \sum_{{\kappa_{\ell'+1},\dots, \kappa_\ell\atop\min(M_j,N_j)=0}\atop
\sum_{j=1}^\ell M_j =\sum_{j=1}^\ell N_j} 
\prod_{i=\ell'+1}^{\ell}  A_i(M_i+\kappa_i) B_i(N_i+\kappa_i)X^{M_i+\kappa_i}
\\&&\qquad \qquad \times
\prod_{j=1}^{\ell'}  \bigg[
X^{M_j\beta_j+N_j\alpha_j}\bigg(
    \sum_{K_j }(A_j\cup\{-\beta_j\})(K_j+M_j) B_j (K_j+N_j )  
  X^{K_j} \\&& \qquad 
-  \sum_{K_j }A _j(K_j+M_j)B_j(K_j+N_j )  
  X^{K_j }\\
&&\qquad +\sum_{K_j} A_j(K_j+M_j)  
  (B_j\cup\{-\alpha_j\})(K_j+N_j )X^{K_j}\bigg)\bigg]
 X^{ M_j(1-\beta_j)-N_j\alpha_j}.
\end{eqnarray*}

Now, the critical observations are that 
$$ 
\sum_{K_1,\dots ,K_\ell\atop
{\sum M_j=\sum N_j  \atop 
\min(M_j,N_j)=0}}
\prod_{j=1}^\ell
A_j(K_j+M_j)B_j(K_j+N_j)X^{K_j+M_j}
 = 
\mathcal C(A_1\cup\dots \cup A_\ell, B_1\cup\dots \cup B_\ell ),
$$
as before, and  
\begin{eqnarray*}&&
\prod_{j=1}^{\ell'} (1-X^{1-\alpha_j-\beta_j}) \mathcal C(A_1\cup\dots 
\cup A_\ell \cup\{-\beta_1,\dots, -\beta_{\ell'}\},
B_1\cup\dots \cup  B_\ell \cup \{-\alpha_1,\dots,-\alpha_{\ell'}\})\\
&& \qquad = 
\sum_{J_1, J_2\subset \{1,\dots,{\ell'}\}
\atop J_1\cap J_2 =\emptyset} (-1)^{|J_1|+|J_2|+\ell'} \mathcal C (A\cup -\beta_{J_1},B\cup -\alpha_{J_2}), 
\end{eqnarray*}
where $A=A_1\cup\dots\cup A_\ell$ and $B=B_1\cup\dots \cup B_\ell$.
 
These together imply Theorem 4.
 
 \section{Multiplicities}
 
 \subsection{How many times is a given $\ell$ swap repeated?}
 Now we need to give an accounting of what we have so far. Each time we split $A$ and $B$ up into subsets $A=A_1\cup \dots \cup A_\ell$ and $B=B_1\cup\dots \cup B_\ell$ we accumulate terms that correspond to all swaps of $\alpha_j\in A_j$ and $\beta_j\in B_j$.  For a fixed decomposition of $A$ into $\ell $ subsets we clearly do not get ALL swaps of $\ell$-sized subsets of  $A$ and $B$.  Our solution to this dilemma is that we consider all decompositions of $A$ into $\ell$ disjoint non-empty subsets and similarly for $B$. Then every pair of $\ell$ sized subsets will indeed appear in the swaps. However, now two different decompositions will often lead to the same swap.  So how do we account for the overcounting?
 
 How many times will a given $\ell$-sized swap $S$ for $T$ occur?   This is equivalent to asking how many ways can $A$ be split into $\ell$ subsets where $A_j$ contains $\alpha_j$? If $A$ has $k$ elements then there are $k-\ell$ elements that can be distributed arbitrarily into $\ell$ sets. This can happen 
 in $\ell^{(k-\ell)}$ ways.  Similarly for $B$. Taking into account permutations we end up with
a multiplicity of  $\ell!^2 \ell^{2(k-\ell)}$.
\subsection{How many times does the same $(m,n)$ lead to a solution of a  $(*)$-system?} 
Our original problem is to evaluate 
$$\sum_{m,n\le X} \frac{\tau_A(m)\tau_B(n)}{\sqrt mn} \hat \psi \big( \frac{T}{2\pi}\log \frac mn\big).
$$
Note that if $A=\{\alpha_1,\dots ,\alpha_k\}$ and $B=\{\beta_1,\dots ,\beta_k\}$ then 
$$\tau_A(m)=\sum_{\mu_1\dots \mu_k=m}\mu_1^{-\alpha_1}\dots \mu_k^{-\alpha_k} 
\qquad \tau_B(n)= \sum_{\nu_1\dots\nu_k=n} \nu_1^{-\beta_1} \dots \nu_k^{-\beta_k}.$$
We  split $A$ into $A_1\cup\dots \cup A_\ell$ and $B$ into $B_1\cup\dots \cup B_\ell$; this is equivalent to splitting $\{1,2,\dots,k\}$ into $I_1\cup \dots \cup I_\ell$ where 
$A_i=\{\alpha_i:i\in I_i\}$ and also $\{1,2,\dots,k\}=J_1\cup \dots \cup J_\ell$ where 
$B_j=\{\beta_j:j\in J_j\}$. Then
$\tau_{A_i}(m_i)=\prod_{i'\in I_i}\mu_{i'}^{-\alpha_{i'}}$ and $\tau_{B_j}(n_j)=\prod_{j'\in J_j}\mu_{j'}^{-\beta_{j'}}$.  Now after this splitting we count the $m_i$ and $n_j$ according to our $(*)$-system:
\begin{eqnarray*}
(*_1): N_1 m_1 &=& M_1 n_1 +h_1\\
 &\vdots&  \\
(*_\ell): N_\ell m_\ell &=& M_\ell n_\ell +h_\ell
\end{eqnarray*}
where $M_1\dots M_\ell=N_1\dots N_\ell$.   Now let's say we have a solution of the $(*)$-system as above and let's take a collection of divisors of the $m_i$ and $n_j$. For simplicity, let's suppose that $\mu_i\mid m_i$ and $\nu_j\mid n_j$ for $1\le i,j\le \ell$. Let's write 
$$m_1=\mu_1\hat \mu_1, \dots ,m_\ell=\mu_\ell \hat \mu_\ell \qquad 
n_1=\nu_1\hat\nu_1 , \dots ,n_\ell=\nu_\ell \hat \nu_\ell.$$
The question is: How many ways are there to do this?  If we multiply the $j$th  equation in our system by $\mu_j^*\nu_j^*$, where $\prod_{j=1}^\ell \mu_j \mu_j^*=m$ and $\prod_{j=1}^\ell
\nu_j\nu_j^*=n$, then we have a new equation 
$$\tilde N_j \tilde m_j= \tilde M_j \tilde n_j +\tilde h_j$$
where 
$$\tilde m_j = \mu_j \mu_j^* \qquad \tilde n_j=\nu_j\nu_j^* \qquad \tilde M_j=M_j\mu_j^*
\hat{\nu_j}
\qquad \tilde N_j =N_j \hat \mu_j \nu_j^* \qquad \tilde h_j = h_j \mu_j^*\nu_j^*.$$
If $(\tilde M_j ,\tilde N_j)>1$ then the common factor can be divided out and out of $\tilde h_j$.
  Note that 
  $$\prod_{j=1}^\ell \tilde M_j = \prod_{j=1}^\ell M_j  \mu_j^* \hat \nu_j=mn  \prod_{j=1}^\ell \frac{M_j}{\mu_j\nu_j}=\prod_{j=1}^\ell \tilde N_j. 
  $$
 Thus, we have a new $(*)$ system but it corresponds to exactly the same $m=\mu_1\dots \mu_k$ and $n=\nu_1\dots \nu_k$ as in the old one. The number of ways to construct these $(*)$-systems
 is just the number of ways to compose the $\hat \mu_j$ as products of  the available $\{\mu_{\ell+1},\dots ,\mu_k\}$ and the $\hat \nu_j$ from the $\{\nu_{\ell+1},\dots \nu_k\}$.  But this is exactly 
 $\ell^{k-\ell}$ for the $\hat \mu_j$ and the same for the $\hat\nu_j$. Then we take into account the ordering of the $\mu_1,\dots ,\mu_\ell$ and of the $\nu_1,\dots,\nu_\ell$; this gives a factor of $\ell!^2$.
 In this way we arrive at a multiplicity  $\ell!^2 \ell^{2(k-\ell)}$ for each solution of our $(*)$-system which is the same as the multiplicity counted in the swaps of $\ell$-sets.  
 
 Note that the same argument applies whether any of the $h_j$ are 0 or not.   We need to divide out this multiplicity. 

 This explains the weight factor $w_\ell$ in (\ref{eqn:start}).
 
 \subsection{Conclusion} 
 We have found that 
 $I_{A;B}^\psi(T;X)$ can be conjecturally evaluated by two different methods which produce the same answer. One way is to use the recipe of [CFKRS]. The other way is to let $\ell$ be defined by
 $T^\ell \le X < T^{\ell+1}$. Then partition  $A$ and $B$   into $\ell $ subsets and evaluate  a convolution of $\ell$  shifted divisor sums 
 $$\sum_{m_i\le u_i\atop m_iN_i-n_iM_i=h_i} \tau_{A_i}(m_i) \tau_{B_i}(n_i) $$
 by a conjectural approach that involves the delta-method of [DFI]. 
 A rigorous theorem identifying two Euler products proves that the result of the above agrees with some of the terms arising from the recipe. 
  The terms with all $h_i\ne 0$ correspond to  $\ell$-swap terms from the recipe. The terms with 
 $\ell'$ of the $h_i$ non-zero and $\ell-\ell'$ of the $h_i$ equal to 0 give  $\ell'$ swap-terms.
Finally,  if  we sum over all possible partitions of $A$ and $B$ into non-empty subsets and account for multiplicities we achieve the desired equality between the two approaches. 
 
  A natural direction for further research is to consider other families of L-functions, for example quadratic   Dirichlet L-functions,  and to determine an arithmetic basis for the relevant moment conjectures.

 \end{document}